\documentclass{article}
\usepackage{amsmath, amssymb, amsthm, color, verbatim,enumerate, cancel, comment, float}
\usepackage{graphicx}
\usepackage{hyperref}
\usepackage{tikz}
\usepackage[numbers]{natbib}

\newtheorem{question}{Question}
\newtheorem{theorem}{Theorem}[section]

\newtheorem{corollary}[theorem]{Corollary}
\newtheorem{proposition}[theorem]{Proposition}
\newtheorem{conjecture}[theorem]{Conjecture}

\title{Locally Resolvable BIBDs and Generalized Quadrangles with Ovoids}
\author{Ryan McCulloch\footnote{Ryan McCulloch, Binghamton University, rmccullo1985@gmail.com}}
\date{July 2024}

\begin{document}

\maketitle

\begin{abstract}
  In this note we establish a 1-to-1 correspondence between the class of generalized quadrangles with ovoids and the class of balanced incomplete block designs that posses a non-triangular local resolution system and have the appropriate parameters.  We present a non-triangular local resolution system for a difference family BIBD construction of Sprott. 
\end{abstract}

{\small
\noindent
{\bf MSC2000\,:} Primary 05B25; Secondary 05B05.

\noindent
{\bf Key words\,:} BIBD, balanced incomplete block design, combinatorial design, block design, design, difference family, generalized quadrangle, ovoid, spread.}

\begin{section}{Introduction}

\vskip 12pt

A finite incidence structure $\mathcal{S} = (\mathcal{P},\mathcal{B},I)$ of points and lines is known as a generalized quadrangle, denoted GQ$(s,t)$ with parameters $s$ and $t$, if it satisfies the following three axioms:

(i) Each point is incident with $1 + t \,\,\, (t \geq 1)$ lines and two distinct points are incident with at most one common line.

(ii) Each line is incident with $1 + s \,\,\, (s \geq 1)$ points and two distinct lines are incident with at most one common point.

(iii) If $x$ is a point and $L$ is a line not incident on $x$, then there is a unique pair $(y,M) \in \mathcal{P} \times \mathcal{B}$ for which $x \, \text{I} \, M \, \text{I} \, y \, \text{I} \, L$.

When $s=t$ we say that the GQ$(s,s)$ has order $s$.  We will use the notation in \cite{GQBook} when describing examples of known generalized quadrangles.

Let $|\mathcal{P}| = v$ and $|\mathcal{B}| = b$.  One can show that 

$$v = (1+s)(1+st), \; \; \; \; \; \; \; \;\; \; \; \; \; b=(1+t)(1+st).$$

Given points $x,y \in \mathcal{P}$, we say that $x$ and $y$ are collinear, and use the notation $x \sim y$ to mean that there is some $L \in \mathcal{B}$ so that $x \, \text{I} \, L \, \text{I} \, y$.

An ovoid, $O$, of a generalized quadrangle $(\mathcal{P},\mathcal{B},I)$ is defined to be a set of points in $\mathcal{P}$ such that every line in $\mathcal{B}$ is incident with exactly one point of $O$.

Not every generalized quadrangle possesses an ovoid.  If a GQ$(s,t)$ does posses an ovoid, $O$, then we have that $|O| = 1 + st$.

\vspace{8 pt}

A $t$-$(v,k,\lambda)$ design consists of a pair $(\mathcal{B}, \mathcal{P})$ where $\mathcal{B}$ is a family of $k$-subsets, called blocks, of a $v$-set of points $\mathcal{P}$ such that every $t$-subset of $\mathcal{P}$ is contained in exactly $\lambda$ blocks.  When $t=2$ and $k < v$, such a design is known as a balanced incomplete block design, or BIBD.  When $k > 2$, the BIBD is said to be nontrivial.  We use the notation BIBD$(v,k,\lambda)$ to refer to a BIBD with parameters $v,k,\lambda$.

For a BIBD, let $|\mathcal{B}| = b$ and let $r$ be the number of blocks in which a point occurs.  The values of $b$ and $r$ can be determined from the other parameters via:

$$vr = bk, \; \; \; \; \; \; \; \;\; \; \; \; \; r(k-1) = \lambda(v-1).$$

A BIBD is said to be \textit{resolvable} if the block set can be partitioned into sets each of which is a partition of the point set.  These sets are called \textit{parallel classes}.  A partition of the blocks into parallel classes is called a \textit{resolution} of the BIBD.

\vskip 15pt

A BIBD with point set $\mathcal{P}$ is said to be \textit{locally resolvable at a point }$p \in \mathcal{P}$ if the family of blocks that contain $p$ can be partitioned into sets so that for each set $S$ in the partition, the set $S' = \{ b - \{p\} \,\, | \,\, b \in S \}$ is a partition of $\mathcal{P} - \{p\}$.  A set $S$ in the partition is called a \textit{parallel class of the BIBD about }$p$.  A partition of the blocks that contain $p$ into parallel classes about $p$ is called a \textit{local resolution of the BIBD about }$p$.  

A BIBD with point set $\mathcal{P}$ is said to be \textit{locally resolvable} if it is locally resolvable at each point $p$ in $\mathcal{P}$.  We define a \textit{local resolution system} of a locally resolvable BIBD to be a collection of local resolutions about $p$ for every $p \in \mathcal{P}$.

A local resolution system of a locally resolvable BIBD is said to be \textit{non-triangular} if it has the following property:

For any three distinct blocks $b, c, d \in \mathcal{B}$, if $b$ and $c$ are in a common parallel class about $p$, and if $b$ and $d$ are in a common parallel class about $q$, for some points $p \neq q$, then $c$ and $d$ are not in a common parallel class about $r$ for any point $r \in \mathcal{P}$.  

Note that distinctness here simply means that $b$, $c$, and $d$ are distinct members of the family $\mathcal{B}$.  The family $\mathcal{B}$ may very well be a multiset that contains repeated blocks, and we consider those to be ``distinct'' .

\end{section}

\begin{section}{Main Theorem}

\begin{theorem}
Let $s > 1$ and $t > 1$ be integers.  
\begin{enumerate}
\item Let $\mathcal{X} = (\mathcal{P},\mathcal{B}, I)$ be a GQ$(s,t)$ that possesses an ovoid $O$.  Let $\mathcal{P}' = O$ and define a a family of blocks $\mathcal{B}'$ as follows: for each $x \in \mathcal{P} - O$, $b_x = \{p \in O \,\, | \,\, x \sim p \}$ is a block in $\mathcal{B}'$.  Given a point $p \in O$, let $l_1, \dots , l_{t+1}$ be the lines of $\mathcal{B}$ incident with $p$; and for each such $l_i$, let $S_{p,i} = \{ b_x \,\, | \,\, x \in \mathcal{P} - O \text{ and } x \text{ I } l_i \}$.  Let $X_p = \{ S_{p,i} \,\, | \,\, 1 \leq i \leq t+1 \}$ and let $\mathcal{C} = \{ X_p \,\, | \,\, p \in O \}$.  Then $\mathcal{Y} = (\mathcal{P}', \mathcal{B}')$ is a locally resolvable BIBD$(1+st,1+t,1+t)$ with $\mathcal{C}$ a non-triangular local resolution system.  Denote $\mathcal{N}(\mathcal{X},O) = (\mathcal{Y},\mathcal{C})$.
\item Suppose $\mathcal{X} =(\mathcal{P}, \mathcal{B})$ is a locally resolvable BIBD$(1+st,1+t,1+t)$ with $\mathcal{C} = \{ X_p \,\, | \,\, p \in \mathcal{P} \}$ a non-triangular local resolution system.  Let $O = \mathcal{P}$ and let $\mathcal{P}' = \mathcal{P} \cup \mathcal{B}$.  Define a line set $\mathcal{B}'$ so that for each $p \in \mathcal{P}$ and each $S \in X_p$, $S \in \mathcal{B}'$.  The incidence relation I is defined as follows: for $p \in \mathcal{P}$ and $S \in \mathcal{B}'$, $p \text{ I } S$ if and only if $S \in X_p$; and for $b \in \mathcal{B}$ and $S \in \mathcal{B}'$, $b \text{ I } S$ if and only if $b \in S$.  Then $\mathcal{Y} = (\mathcal{P}',\mathcal{B}',I)$ is a GQ$(s,t)$ with ovoid $O$.  Denote $\mathcal{M}(\mathcal{X},\mathcal{C}) = (\mathcal{Y},O)$.
\item Let $\mathcal{X}$ be a GQ$(s,t)$ that possesses an ovoid $O$ and let $\mathcal{Y}$ be a locally resolvable BIBD$(1+st,1+t,1+t)$ with $\mathcal{C}$ a non-triangular local resolution system.  Let $\mathcal{N}(\mathcal{X},O)$ and $\mathcal{M}(\mathcal{Y},\mathcal{C})$ be defined as in the previous two items.  Then $\mathcal{M}(\mathcal{N}(\mathcal{X},O)) = (\mathcal{X},O)$ and $\mathcal{N}(\mathcal{M}(\mathcal{Y},\mathcal{C})) = (\mathcal{Y},\mathcal{C})$.
\end{enumerate}
\end{theorem}

\begin{proof}
\textit{Item 1.}
So $|\mathcal{P}'| = 1+st = |O|$, and each block in $\mathcal{B}'$ has size $1+t$.  If $p,q \in \mathcal{P}'$ with $p \neq q$, then $|\{ x \in \mathcal{P} - O \,\, | \,\, x \sim p \text{ and } x \sim q \}| = 1+t$, and so we have $1+t$ blocks that contain $\{p,q\}$.  Note that given a point $p \in O$, a line $l_i \text{ I } p$, and points $x,y \in \mathcal{P} - O$, $x \neq y$, with $x \text{ I } l_i$ and $y \text{ I } l_i$, we have that $b_x \cap b_y = \{ p \}$ (otherwise we create a triangle).  And, furthermore, for any $q \in O$, $q \neq p$, there is some $z \text{ I } l_i$ with $z \sim q$.  Hence we have that $S_{p,i}' = \{ b_x - \{p\} \,\, | \,\, b_x \in S_{p,i} \}$ is a partition of $O - \{p\}$.

We show that $\mathcal{C}$ is non-triangular.  Let $b_x$, $c_y$, $d_z$ be distinct members of $\mathcal{B}'$, and so $x,y,z$ are distinct points of $\mathcal{P} - O$.  Suppose $b_x$ and $c_y$ are in $S_{p,i}$ for some line $l_i$ and some point $p \in O$ with $p \text{ I } l_i \text{ I } x$ and $p \text{ I } l_i \text{ I } y$.  Suppose $b_x$ and $d_z$ are in $S_{q,j}$ for some line $l_j$ and some point $q \in O$ with $q \text{ I } l_j \text{ I } x$ and $q \text{ I } l_j \text{ I } z$, and suppose $p \neq q$.  Since $p$ and $q$ are ovoid points, we have that $l_i$ and $l_j$ are distinct.  Suppose by way of contradiction that $c_y$ and $d_z$ are in $S_{r,k}$ for some line $l_k$ and some point $r \in O$ with $r \text{ I } l_k \text{ I } y$ and $r \text{ I } l_k \text{ I } z$.  We show that $r$ is distinct from $p$ and $q$.  Suppose, say, that $r=p$.  Then $l_i=l_k$.   And so $x,z$ are both incident with two common lines, $l_i$ and $l_j$, a contradiction.  So $p,q,r$ are distinct and hence $l_i$, $l_j$ and $l_k$ are distinct.  But then $x \text{ I } l_i \text{ I } y$, $x \text{ I } l_j \text{ I } z$, and $y \text{ I } l_k \text{ I } z$, a contradiction.

\textit{Proof of item 2.}
It is clear by construction that every line in $\mathcal{B}'$ is incident with exactly one point of $O$.  Since for each $p \in \mathcal{P}$, $X_p$ is a partition of the family of blocks containing $p$, we have that two distinct lines cannot both be incident with a point of $\mathcal{P}$ and a point of $\mathcal{B}$.  And if two distinct lines were both incident with two distinct points of $\mathcal{B}$, say $b$ and $c$, then $\{p,q\} \subseteq b \cap c$ for distinct $p,q \in \mathcal{P}$, contradicting the fact that for each $S \in X_p$, $S' = \{ b - \{p\} \,\, | \,\, b \in S \}$ is a partition of $\mathcal{P} - \{p\}$.  Equivalently, two distinct points cannot be incident with more than one common line.  

Note that for $p \in \mathcal{P}$, $p$ occurs in exactly $r = \dfrac{(1+t)st}{t} = (1+t)s$ blocks of the BIBD.  Since for each $S \in X_p$, $S' = \{ b - \{p\} \,\, | \,\, b \in S \}$ is a partition of $\mathcal{P} - \{p\}$, and since for each $o \in \mathcal{P} - \{p\}$, $\{o,p\}$ is contained in exactly $1+t$ blocks of $\mathcal{B}$, we have that $|X_p| = 1+t$.  It follows that each $S \in X_p$ has size $s$.  If $p \in \mathcal{P}$, then $p$ is incident with exactly $|X_p| = 1+t$ lines in $\mathcal{B}'$.  And if $b \in \mathcal{B}$, then $b$ is incident with exactly $1+t$ lines in $\mathcal{B}'$ since $|\{ X_q \,\, | \,\, q \in b\}|=1+t$.  If $S \in \mathcal{B}'$, then $S$ is incident with exactly $1 + |S| = 1+s$ points.  We now work to establish the third generalized quadrangle axiom.

Let $S \in \mathcal{B}'$ be arbitrary and $o$ be such that $S \in X_o$.  

Case 1 is to consider $p \in \mathcal{P}$ with $p \neq o$.  Since $S' = \{ b - \{o\} \,\, | \,\, b \in S \}$ is a partition of $\mathcal{P} - \{o\}$, there is a unique $b \in S$ with $p \in b$.  And since $X_p$ is a partition of the family of blocks containing $p$, there is a unique $M \in X_p$ with $b \in M$.   Hence $p \text{ I } M \text{ I } b \text{ I } S$.

Case 2(a) is to consider $b \in \mathcal{B}$ where $b \in S'$ for some $S' \in X_o$ with $S \neq S'$.  So $b \text{ I } S' \text{ I } o \text{ I } S$ and $S'$ is the unique line in $X_o$ incident with $b$.  Now if there is some $o' \in \mathcal{P}$ with $o' \neq o$ and there is $T \in X_{o'}$ with $b \text{ I } T \text{ I } b' \text{ I } S$ for some $b' \in S$, then $o \in b$ and $o \in b'$ with $b,b' \in T \in X_{o'}$, contradicting the fact that  $T' = \{ b - \{o'\} \,\, | \,\, b \in T \}$ is a partition of $\mathcal{P} - \{o'\}$.

Finally, case 2(b) is to consider $b \in \mathcal{B}$ where $b$ is not incident with any line in $X_o$.  So $b = \{o_1,...,o_{1+t}\}$ with each $o_i \neq o$.  For each $o_i \in b$, let $S_i$ be the unique line in $X_{o_i}$ incident $b$.  Note that each $S_i$ contains exactly one block that contains $o$.  Now if there was an $(S_i,b')$ pair with $b' \in S \cap S_i$ and if there was an $(S_j,b'')$ pair with $b'' \in S \cap S_j$, and with $S_i \neq S_j$, then this would contradict the fact that $\mathcal{C}$ is non-triangular.  Since the number of $S_i$'s is $1+t$, and $|X_o|=1+t$ and $X_o$ is a partition of the family of blocks that contain $o$, it follows that there exists some $(S_i,b')$ which is the unique pair with $b' \in S \cap S_i$, and $b \text{ I } S_i \text{ I } b' \text{ I } S$.

\textit{Proof of item 3 is clear.}
\end{proof}

\begin{corollary}
Locally resolvable BIBDs that posses a non-triangular local resolution system and with the following $(v, b, r, k,\lambda)$ parameters arise from ovoids in known generalized quadrangles: 
\begin{enumerate}
\item $(q^2 + 1, (q^2+1)q, (q+1)q, q+1, q+1)$ where $q$ is a prime power.
\item $(q^3 + 1, (q^3+1)q^2, (q+1)q^2, q+1, q+1)$ where $q$ is a prime power.
\item $(q^2, (q+1)q^2, (q+1)q, q, q)$ where $q$ is a prime power.
\item $(q^2, (q-1)q^2, (q-1)(q+2), q+2, q+2)$ where $q$ is a power of $2$.
\end{enumerate}
\end{corollary}

\begin{proof}
For item 1, take $Q(4,q)$ which is a GQ$(q,q)$ and is known to possess ovoids.  For item 2, take $H(3,q)$ which is a GQ$(q^2,q)$ and is known to possess ovoids.  For item 3, take the dual of $P(W(q),x)$ for any point $x$ of $W(q)$.  Such a GQ$(q+1,q-1)$ is known to always possess ovoids.  For item 4, take $P(Q(4,q),x)$ where $q$ is a power of 2 and $x$ is a point in an ovoid of $Q(4,q)$.  It is known that all of the points of $Q(4,q)$ are regular when $q$ is a power of 2, and it is known that $P(Q(4,q),x)$ possesses an ovoid when $x$ is a point in an ovoid of $Q(4,q)$.
\end{proof}

\end{section}

\begin{section}{Examples}
We begin this section with a BIBD coming from a difference family construction due to Sprott\cite[Theorem 2.1]{Sprott1956} (see \cite{AbBur} for terminology related to difference families).

Let $p$ be a prime and let $m$ and $\lambda$ be such that $m(\lambda-1) = p^a$.  Let $x$ be a primitive element of $GF(p^a)$.  We define $D$, a set of base blocks, by $D = \{ (0, x^i, x^{i+m}, x^{i+2m}, \dots, x^{i+(\lambda-2)m}) \,\, | \,\, 0 \leq i \leq m-1 \}$.  Sprott shows that $D$ satisfies the conditions for being a difference family.  We shall refer to the associated BIBD with parameters $(v,k,\lambda)=(p^a,\lambda,\lambda)$ as $\text{Sprott}(p^a,\lambda)$.

\begin{proposition}
Let $q$ be a power of $2$ and let $x$ be a primitive element of $GF(q^2)$.  Let $\mathcal{X} = \text{Sprott}(q^2,q+2)$.  Let $S_{0,0} = \{ (0, x^i, x^{i+(q-1)},x^{i+2(q-1)}, \dots, \\ x^{i+q(q-1)}) \,\, | \,\, 0 \leq i \leq q-2 \}$ be the set of base blocks.  For $0 \leq j \leq q$, let $S_{0,j+1} = \{ x^{j+(q+1)i} \cdot (0, 1, x^{q-1}+1,x^{2(q-1)}+1, \dots, x^{q(q-1)}+1 ) \,\, | \,\, 0 \leq i \leq q-2 \}$.  Let $X_0 = \{ S_{0,i} \,\, | \,\, 0 \leq i \leq q+1 \}$.  For $v \in GF(q^2)$, let $X_v$ be the translation of $X_0$ by $v$ (i.e. $X_0 + v$).  Let $\mathcal{C} = \{ X_v \,\, | \,\, v \in GF(q^2) \}$.  Then $\mathcal{C}$ is the uniquely determined local resolution system for $\mathcal{X}$, and $\mathcal{C}$ is non-triangular.
\end{proposition}

\begin{proof}
Exercise.
\end{proof}

We verified for some $q$ a power of $2$ that the GQ$(q-1,q+1)$ resulting from Sprott$(q^2,q+2)$ is isomorphic to the Payne GQ $P(W(q),x)$.

One can see that for $q$ a power of a prime ($q$ even or odd), $\text{Sprott}(q^2,q)$ consists of $q$ copies of a Desarguesian affine plane of order $q$.  In this case, a local resolution system clearly exists (put multiples of a block into different parallel classes).  We were able to find, through a computer search, a non-triangular local resolution system for the first few Desarguesian affine planes of order $q$ (up to at least $q=16$), and we found that the resulting GQs are isomorphic to the dual of the Payne GQ $P(W(q),x)$.  We wonder if two different non-triangular local resolution systems (for the same BIBD) could result in non-isomorphic GQs.  

\begin{proposition} Let $s > 1$ and $t > 1$ be integers.  A BIBD $\mathcal{X}$ consisting of $1+t$ copies of a $BIBD(1+st,1+t,1)$ possesses a non-triangular local resolution system if and only if a GQ$(s,t) = (\mathcal{P},\mathcal{B})$ possesses an ovoid $O$ so that for each point $x \in \mathcal{P} - O$, there exists a point $y \in \mathcal{P} - O$, $x$ and $y$ not collinear, so that the pair $(x,y)$ is regular and so that the trace of $x$ and $y$ is contained in $O$.
\end{proposition}

In fact each block of the ovoid is the trace of a regular pair $(x,y)$, with the multiplicity of the block equal to $1+t$.

Proposition 3.2 demonstrates that for such an isomorphism class of BIBD, the assumption of the existence of a non-triangular local resolution system (without explicitly defining it) can lead to restrictions on the structure of the resulting GQ.  

We verified, for the first few values of $q$, that such ovoids as in Proposition 3.2 do exist in the dual of the Payne GQ $P(W(q),x)$, and we also used a computer to find all such ovoids (there aren't too many).  Each such ovoid resulted in a BIBD that consists of $q$ copies of a Desarguesian affine plane of order $q$.

\begin{conjecture}
A BIBD that consists of $q$ copies of a non-Desarguesian affine plane of order $q$ does not possess a non-triangular local resolution system.
\end{conjecture}

This conjecture may indeed be false, and if so it would be quite interesting.  One can look for ovoids as in Proposition 3.2 existing in the duals of some of the unusual Payne GQs $P(T_2(O),x)$ (we only checked in the duals of some of the $P(W(q),x)$).  One can try each of the $88$ affine planes of order $16$ and use a computer to search for a non-triangular local resolution system.  (We note that the BIBD table in \cite{MatRos} indicates that there are $189$ affine planes of order 16.  This is incorrect.  There are, in fact, 88 non-isomorphic affine planes of order 16 coming from the 22 known projective planes.)

\end{section}

\begin{section}{Conclusion}

We conclude with two questions based on our main theorem.

\begin{question}
Does a known GQ$(s,t)$ with some ovoid $O$ result in a previously unknown BIBD$(1+st,1+t,1+t)$?
\end{question}

\begin{question}
Given a known BIBD$(1+st,1+t,1+t)$, can one construct a non-triangular local resolution system for the BIBD that results in a new ovoid for a known GQ$(s,t)$, or, perhaps, that results in a new GQ$(s,t)$ with ovoid?
\end{question}

\end{section}

\bigskip\noindent {\bf Acknowledgements.} The author is grateful to G. Eric Moorhouse, to Eric Swartz, and to the anonymous referee for helpful suggestions during the writing of this paper.  The author also thanks Austin C. Bussey.

\bibliographystyle{plainnat}
\bibliography{ref}

\end{document}